\documentclass[reqno]{amsart}

\usepackage{amsmath}
\usepackage{amsfonts}
\usepackage{amssymb,enumerate}
\usepackage{amsthm}
\usepackage[all]{xy}
\usepackage{rotating}
\usepackage{graphicx}
\usepackage[colorlinks=true, linkcolor=blue, anchorcolor=blue, citecolor=blue, filecolor=blue, menucolor= blue, urlcolor=blue]{hyperref}

\theoremstyle{plain}
\newtheorem{lem}{Lemma}[section]
\newtheorem{cor}[lem]{Corollary}
\newtheorem{prop}[lem]{Proposition}
\newtheorem{thm}[lem]{Theorem}

\theoremstyle{definition}
\newtheorem{defn}[lem]{Definition}
\newtheorem{ex}[lem]{Example}

\newtheorem{disc}[lem]{Remark}

\newtheorem{notation}[lem]{Notation}

\newtheorem*{convention*}{Convention}

\newcommand{\card}{\operatorname{card}}

\renewcommand{\leq}{\leqslant}

\numberwithin{equation}{lem}

\newcommand{\al}{\alpha}
\newcommand{\be}{\beta}

\newcommand {\PP}{\mathbb{P}}
\newcommand {\QQ}{\mathbb{Q}}
\newcommand {\AAA}{\mathbb{A}}

\newcommand{\qbinom}[2]{{\left[\begin{smallmatrix}  #1  \\ #2  \end{smallmatrix}\right]_q}}

\newcommand{\cal}{\mathcal}
\newcommand {\FF}{\mathbb{F}}
\newcommand{\pa}{{\partial}}
\newcommand{\Ann}{{\rm Ann}}

\begin{document}

\bibliographystyle{amsplain}

\title[Determinants of incidence and Hessian matrices]{Determinants of incidence and Hessian matrices arising from the vector space lattice}

\author{Saeed Nasseh}
\address{Department of Mathematical Sciences\\
Georgia Southern University\\
Statesboro, GA 30460, USA}
\email{snasseh@georgiasouthern.edu}
\author{Alexandra Seceleanu}
\address{Department of Mathematics\\
University of Nebraska\\
Lincoln, NE 68588-0130, USA}
\email{aseceleanu@unl.edu}
\author{Junzo Watanabe}
\address{Department of Mathematics\\
Tokai University \\
 Hiratsuka 259-1292, Japan}
\email{watanabe.junzo@tokai-u.jp}


\date{\today}



\begin{abstract}
Let $\mathcal{V}=\bigsqcup _{i=0}^n\mathcal{V}_i$ be the lattice of subspaces of the $n$-dimensional vector space 
over the finite field  $\FF _q$ and  
let $\mathcal{A}$ be the graded Gorenstein algebra defined over $\QQ$ which has $\mathcal{V}$  
as a $\QQ$ basis.  Let  $F$  be the Macaulay dual generator for $\mathcal{A}$.  
We compute explicitly the Hessian determinant  $|\frac{\pa ^2F}{\pa X_i \pa X_j}|$ evaluated at the point 
$X_1 = X_2 = \cdots = X_N=1$ and relate it to the determinant of the incidence matrix between $\mathcal{V}_1$ and 
$\mathcal{V}_{n-1}$. Our exploration is motivated by the fact that both of these matrices arise naturally in the study of the 
Sperner property of the lattice  and the  Lefschetz property for the graded Artinian Gorenstein algebra associated to it.   
\end{abstract}

\maketitle

\section{Introduction}


Let $P$ be a poset with a rank function $\rho:P \to \mathbb{N}$.
Then $P$ decomposes into a disjoint union of the level sets, namely $P=\bigsqcup _{i=0}^cP_i$, where $P_i=\{x \in P \mid \rho(x)=i\}$. We say that $P$ has the  Sperner property if the maximum size of antichains of $P$ is equal to the maximum of the rank numbers $|P_i|$. Some of the basic examples of finite ranked posets known to have the Sperner property  are the Boolean lattice, the divisor lattice, and the vector space lattice over a finite field. One way to show that the Sperner property holds for the vector space lattice is as consequence of the fact that certain incidence matrices have full rank as illustrated in \cite[Theorem 1.83]{HMMNWW}. We will say that a ranked poset with a symmetric sequence of rank numbers has the strong Lefschetz property if the incidence matrices between every pair of symmetric level sets are invertible. This implies the Sperner property for posets with symmetric sequence of rank numbers by \cite[Lemmas 1.51, 1.52]{HMMNWW}. For  the vector space lattice, the fact  that it has the strong Lefschetz property follows from a result of Kantor \cite{Ka}.  There are several other ways to show that the vector space lattice has the Sperner property; the reader may consult \cite{GK}  for details.

It is remarkable that some posets with a rank function can be  vector space bases for some graded Artinian algebras over a field in such a way that the 
multiplication of the algebra is compatible with the incidence structure of the poset. For example the Boolean lattice $2^{\{x_1, \cdots, x_n\}}$ can be the basis 
for  the algebra \[K[x_1,x_2, \cdots, x_n]/(x_1^2, x_2 ^2, \cdots, x_n^2).\] 
Recently Maeno and Numata~\cite{MN} succeeded in constructing a family of algebras over a field for which vector space lattices are the bases. 
To explain briefly their construction,  let $\FF _q$ be the finite field with $q$ elements, 
$V=\FF_q ^n$  the $n$-dimensional vector space and  $\cal{V}=\bigsqcup _{i=0}^n\cal{V}_i$ the vector space lattice with rank decomposition.   
Introduce as many variables as the number of the one dimensional subspaces of $V$ and then define the form 
\[F=\sum x_{i_1}x_{i_2}\cdots x_{i_n},\]
where the indices run over the combinations such that span 
$\langle x_{i_1}, x_{i_2}, \cdots, x_{i_n}\rangle$ is the whole space $V$.  
(A variable like $x_i$ represents a one dimensional vector subspace of $V$ and distinct variables represent distinct spaces.) 
Let $R=K[x_1, \cdots, x_N]$ be the polynomial ring in $N$ variables, where $N$ is the number of one dimensional subspaces of $V$. 
(Note $K$ is any field and should not be confused with  $\FF _q$.)
Set $\cal{A}=R/\Ann (F)$.  The Artinian algebra $\cal{A}$ has the Hilbert function displayed below
\[\left({n \atopwithdelims[] 0}_q, {n \atopwithdelims[] 1}_q, \cdots, {n \atopwithdelims[] n}_q \right).\]
An explicit formula for  ${n \atopwithdelims[] i}_q$  is given in the beginning of section 2.
Every  monomial $m$ in $\cal{A}$ represents a vector subspace in $V$ of the dimension which is equal to the degree of $m$.   

We are interested in the Hessian determinant $|\frac{\pa ^2F}{\pa x_i \pa x_j}|$ 
of $F$ evaluated at 
$x_1 = \cdots = x_N = 1$.  The motivation for it is as follows: 
 it is proven in \cite{refMW} that the non-vanishing of the Hessian, together with the non-vanishing of the  higher Hessians of the Macaulay dual generator $F$ (i.e., $| \frac{\pa ^{2k}F}{\pa x_{i_1} \ldots \pa x_{i_k}\pa x_{j_1} \ldots \pa x_{j_k}}|$)  is equivalent to the strong Lefschetz property for the  Gorenstein  algebra (Definition \ref{def Lefschetz}), which ensures the Sperner property of the poset.   
This suggests that a connection exists between the higher Hessians evaluated at a certain point $(x_i)$ and the determinants 
of the incidence matrices for the vector space lattice. (Recall that the first Hessian of $F$ is the Hessian in the usual sense.)
Our main result is Theorem~\ref{main2}, where we make  explicit the relation between the Hessian matrix and the incidence matrix of 
the vector space lattice and we derive from it a closed formula in Corollary~\ref{det Hessian}  for the Hessian of $F$ 
evaluated at $x_1 = \cdots = x_N=1$.   
  
In the literature, efforts have been made to obtain the Smith normal form of incidence matrices for various posets (\cite{Si}). In particular the Smith normal form for the incidence matrix between the sets $\cal{V}_1$ and $\cal{V}_{n-1}$  is obtained by Xiang~\cite{Xi}. The determinant itself is much easier to obtain; it is enough to notice that 
\[A\;{}^{T}A = (N-\lambda )I + \lambda J,\] 
where $I$ is the $N \times N$ identity matrix and $J$ is the matrix with all $1$ as entries. 
This is due to Xiang~\cite[(1.1)]{Xi}. In this paper we reproduce a proof for it since this does not seem to be well known among the commutative algebraists (Theorem \ref{main1} (c)).     

Computations similar in spirit  have been performed for evaluating the  
determinants of all incidence matrices of the  Boolean lattice 
in \cite{Pr} and \cite{HW}, obtaining explicit and recursive formulas 
respectively. For a comprehensive survey of determinant evaluations and 
their many applications see \cite{Kr}.

Our paper is organized as follows: in Section \ref{sect2} we gather useful properties of the vector space lattice, 
focusing  on enumerative results. 
In Section \ref{sect3} we carry out our computation of the determinant of the incidence matrix between the first  level set and the 
$(n-1)$st level set. 
In Section \ref{sect4} we recall Maeno--Numata's construction of the graded Artinian Gorenstein  algebra $\cal{A}$ associated with the 
vector space lattice as introduced in \cite{MN}. We explicitly describe the Hessian matrix of the Macaulay dual generator of $\cal{A}$ and 
we compute the Hessian determinant.   Furthermore, we show that the same method can be used to obtain the determinant 
for the multiplication map 
$\times L : \cal{A}_1 \to \cal{A}_{n-1}$, where $L:=\sum _{j=1}^Nx_j$, and the matrix is written with respect to the monomial bases.

\section{The vector space lattice}\label{sect2}
Throughout this paper, let $\FF$ be the finite field  with $q$ elements and let $n$ be a positive integer.

\begin{defn}
The \emph{vector space lattice} on $\FF^n$, denoted ${\mathcal V}(n,q)$, is
the set of all subspaces of $\FF^n$ naturally ordered by inclusion.
Note that ${\mathcal V}(n,q)$ is a poset with the \emph{rank function} $\rho$ defined by $\rho(W)=\dim_{\FF}(W)$, for each $W\in {\mathcal V}(n,q)$.
This gives rise to the \emph{rank decomposition}   ${\mathcal V}(n,q)=
\bigsqcup _{j=0}^n{\mathcal V}_j$  into \emph{level sets}
${\mathcal V}_j:=\{ W \in {\mathcal V}(n,q) \ \mid\ \dim_{\FF}(W) = j\}$.
\end{defn}

Using the notation  $[i]=(q^i-1)/(q-1)$  for the $q$-integers, we recall the formula for the sizes of the level sets in  the vector space lattice (see \cite[Proposition 1.81]{HMMNWW}):
\begin{equation*}\label{levelcount}
\card({\mathcal V}_j)=\left[\begin{matrix} n \\ j  \end{matrix}  \right]_q, \mbox{ where }
\left[\begin{matrix} n \\ j  \end{matrix}  \right]_q=
\begin{cases}
\frac{[n][n-1]\ldots [n-j+1]}{[j][j-1]\ldots [1]} & (0\leq j\leq n)\\
0 & (j<0 \mbox{ or } j>n).
\end{cases}
\end{equation*}

Let $G(n,m)$ denote the Grassmannian variety  of $m$-dimensional subspaces of an $n$-dimensional vector space. One reason for studying the vector space lattice ${\cal V}$ is that  each level set ${\cal V}_j$ may be regarded as the set of rational points of the Grassmannian variety $G(n, j)$ corresponding to a finite vector space.
In our work we routinely identify the set ${\cal V}_j$ as the collection of
$n \times j$ matrices in echelon form with entries in $\FF$.
For example, for $n=4$ and $j=2$, the set ${\cal V}_j$ is in one-one
correspondence with the set
$$
\begin{aligned}
\Big\{\begin{pmatrix} 1&0&*&* \\ 0&1&*&* \end{pmatrix},
\begin{pmatrix} 1&*&0&* \\ 0&0&1&* \end{pmatrix},&
\begin{pmatrix} 1&*&*&0 \\ 0&0&0&1 \end{pmatrix},\\
\begin{pmatrix} 0&1&0&* \\ 0&0&1&* \end{pmatrix},&
\begin{pmatrix} 0&1&*&0 \\ 0&0&0&1 \end{pmatrix},
\begin{pmatrix} 0&0&1&0 \\ 0&0&0&1 \end{pmatrix}\Big\}
\end{aligned}
$$
where each echelon form corresponds to the subspace spanned by the rows of the respective matrix.

For $W \in {\cal V}(n, q)$, define the \emph{dual subspace} $W^\perp \in {\cal V}(n, q)$ by
$$W^\perp=\left\{ w\in \FF^n\ \mid\ \sum_{i=1}^n v_iw_i=0,\ \text{for all}\ v\in W\right\}.$$
The map ${\cal V}(n, q)\to {\cal V}(n, q)$ given by $W\mapsto W^\perp$  is an inclusion-reversing bijection meaning that it satisfies the condition: $U \subseteq W$ if and only if $W^\perp \subseteq U^\perp$.

 Focusing on the level sets of elements of rank 1 and $n-1$, respectively, the formula for the sizes of the level sets  gives
$\card({\cal V}_1) = \card({\cal V}_{n-1})=\left[\begin{smallmatrix}  n \\ 1 \end{smallmatrix}\right]_q$.
Set $N=\card({\cal V}_1)$
and fix the following notation for elements of the level set   ${\cal V}_1$:
\[{\cal V}_1 =\{v_1, v_2, \ldots , v_N \}.\]
In particular, the set ${\cal V}_1$ is in one-one correspondence with the
rational points of the  projective space $\PP ^{n-1} _{\FF}$.  Thus it will be convenient to regard
${\cal V}_1$ as the set of vectors $(a_1, \ldots, a_n)$ such that
the first nonzero component is $1$.  These vectors are a special case of the echelon matrices described above. Since $\PP ^{n-1}_{\FF} = \PP ^{n-2}_{\FF} \sqcup \AAA ^{n-1}_{\FF}$, we have the identity $N = \left[\begin{smallmatrix}  n \\ 1 \end{smallmatrix}\right]_q= \left[\begin{smallmatrix}  n-1 \\ 1 \end{smallmatrix}\right]_q +q^{n-1}$.

We denote by $v_k ^{\perp}$ the dual space of $\FF v_k$, which  allows us to identify the $(n-1)$-st level set of the vector space lattice with the set of duals of elements of the first level set as follows:
\[{\cal V}_{n-1}=\{v_1 ^{\perp} , v_2 ^{\perp}, \ldots, v_N^{\perp}\}.\]

The following definition introduces the focal point of our attention in this work.
\begin{defn}\label{defA}
The \emph{incidence matrix} $A=(a_{ij})$ for ${\cal V}_1$  and ${\cal V}_{n-1}$ is the $N \times  N$ matrix whose entries are
$$
a_{ij}=
\begin{cases}
1  &(v_i \in v_j^{\perp})\footnote{} \\
0 &(v_i \not \in v_j^{\perp}).
\end{cases}
$$
\end{defn}
\footnotetext{Throughout this article, we write $v_i \in v_j^{\perp}$ rather than $v_i \subset v_j^{\perp}$ because we prefer to think of of $v_i$ as vectors rather than subspaces of $V$, via a canonical identification explained above.}

The first goal of this note is to find a closed formula for the determinant of the incidence matrix $A$. While our vector space lattice is defined over a field of positive characteristic, all of our determinant computations will be performed in characteristic zero. This is to preserve the enumerative properties of the entries in our matrices. Note that the truly meaningful invariant of the incidence structure between ${\cal V}_1$  and ${\cal V}_{n-1}$ is in fact the absolute value of this determinant, denoted $\left| \det A \right|$, since this is preserved under permuting the order of the elements in ${\cal V}_1$  and ${\cal V}_{n-1}$.

We begin by describing the incidence matrix in a concrete example.
\begin{ex}\label{ex1}
Let $q=2$ and $n=3$. In this case we have $N=7$. Then ${\mathcal V}_1 =\{v_1, v_2, \ldots , v_7 \}$, in which
\begin{gather}
\begin{align}
v_1=(0,0,1)&
&v_2=(0,1,0)&
&v_3=(0,1,1)&
&v_4=(1,0,0)
\notag
\end{align}
\\
\begin{align}
&v_5=(1,0,1)&
&v_6=(1,1,0)&
&v_7=(1,1,1).&\notag
\end{align}
\end{gather}
Now we have ${\mathcal V}_2={\mathcal V}_1^{\perp}=\{u_1, u_2, \ldots , u_7\}$, where
\begin{align}
u_1:=v_1^{\perp}=\begin{pmatrix} 1&0&0 \\ 0&1&0 \end{pmatrix}&
&u_2:=v_2^{\perp}=\begin{pmatrix} 1&0&0 \\ 0&0&1 \end{pmatrix}&
&u_3:=v_3^{\perp}=\begin{pmatrix} 1&0&0 \\ 0&1&1 \end{pmatrix}
\notag
\\
u_4:=v_4^{\perp}=\begin{pmatrix} 0&1&0 \\ 0&0&1 \end{pmatrix}&
&u_5:=v_5^{\perp}=\begin{pmatrix} 1&0&1 \\ 0&1&0 \end{pmatrix}&
&u_6:=v_6^{\perp}=\begin{pmatrix} 1&1&0 \\ 0&0&1 \end{pmatrix}\notag
\\
&
&u_7:=v_7^{\perp}=\begin{pmatrix} 1&0&1 \\ 0&1&1 \end{pmatrix}.&
&\notag
\end{align}
Therefore we can compute the incidence matrix $A$ as displayed below, which gives $\det(A)=- 3\cdot 2^3$ and
$$
A=\begin{pmatrix} 0&1&0&1&0&1&0 \\ 1&0&0&1&1&0&0 \\ 0&0&1&1&0&0&1 \\ 1&1&1&0&0&0&0 \\ 0&1&0&0&1&0&1\\ 1&0&0&0&0&1&1\\ 0&0&1&0&1&1&0\end{pmatrix}.
$$
\end{ex}

For later use in our computations, we record a few enumerative invariants of the lattice ${\cal V}(n,q)$. We employ the notation $\card$ for the cardinality of a finite set.

 \begin{prop}  \label{the number of paths}
 The following enumerative identities hold true:
\begin{enumerate}[\rm(a)]
\item
$\card({\rm GL}(n, \FF))=(q^n-1)(q^n-q^1)(q^n-q^2)\dots (q^n- q^{n-1})$.
\item
The number of  ordered $n$-tuple subsets of $\cal V_1$ which form bases for $\FF^n$
is
\[t_{n,q}= \frac{\card({\rm GL}(n, \FF))}{(q-1)^n}=\left(q^{n(n-1)/2}\right)\left(\prod_{k=1} ^n \left[\begin{matrix}k\\1 \end{matrix}\right]_q\right).\]
\item
The number of   $n$-tuple subsets of $\cal V_1$ which form bases for $\FF^n$
is
\[s_{n,q}= \frac{\card({\rm GL}(n, \FF))}{n!(q-1)^n}=\left(\frac{q^{n(n-1)/2}}{n!}\right)\left(\prod _{k=1} ^n \left[\begin{matrix}k\\1 \end{matrix}\right]_q\right).\]
\item
The number of  ordered $n$-tuple subsets of $\cal V_1$ which form bases for $\FF^n$ and contain a fixed linearly independent ordered subset of size $j$ is
\[t_{n,j,q}=\left(q^{\left(n(n-1)-j(j-1)\right)/2}\right)\left(\prod _{k=1} ^{n-j} \left[\begin{matrix}k\\1 \end{matrix}\right]_q\right).\]
\item
The number of  $n$-tuple subsets of $\cal V_1$ which form bases for $\FF^n$and contain a fixed linearly independent subset of size $j$ is
\[s_{n,j,q}=\left(\frac{q^{(n(n-1)-j(j-1))/2}}{(n-j)!}\right)\left(\prod _{k=1} ^{n-j} \left[\begin{matrix}k\\1 \end{matrix}\right]_q\right).\]
\item
The number of paths in ${\cal V}(n,q)$ from the minimum element to the maximum element in the vector space lattice of $\FF^n$ is equal to
$$p_{n,q}=\prod _{k=1}^n\left[\begin{matrix}k\\1 \end{matrix}\right]_q.$$
\end{enumerate}
\end{prop}

\begin{proof}
(a)  Any nonzero vector can be the first row of an $n \times n$ invertible matrix.
If the first $k$ rows $u_1, \ldots, u_k \in \FF^n$  of an invertible matrix are chosen,
then any vector in $\FF^n \setminus \sum _{i=1}^k \FF u_i$ can be the $(k+1)$-st row for such a matrix.
This proves  the formula inductively for the number of elements in  ${\rm GL}(n, \FF)$.

(b)  We regard such an ordered $n$-tuple of vectors as a matrix $U \in {\rm GL}(n, \FF)$ and we let $u_i$ be the $i$-th row. Then, for each integer $i$,  we may find  a unique  vector $v_{k_i}\in {\cal V}_1$ such that $\FF v_{k_i} = \FF u_i$.  The correspondence $U \mapsto (v_{k_1}, \ldots, v_{k_n})$ is $(q-1)^n :1$, where by $(v_{k_1}, \ldots, v_{k_n})$ we mean the ordered $n$-tuple. This proves that the number of ordered $n$-tuple subsets of $\cal V_1$ which form bases for $\FF^n$  is equal to
\[\frac{(q^n-1)(q^n-q^1)(q^n-q^2)\ldots(q^n-q^{n-1})}{(q-1)^n}.\]
Noting  that \[\frac{q^n-q^k}{q-1}= q^k\frac{q^{n-k}-1}{q-1}=q^k\left[\begin{matrix}n-k\\1 \end{matrix}\right]_q,\]
we may rewrite the expression above as the claimed formula.

(c) This is easily deduced by observing that the correspondence between the ordered tuples of part (b) and the unordered ones is $n! :1$.

(d)  We regard such an ordered $n$-tuple of vectors as a matrix $U \in {\rm GL}(n, \FF)$, where the first $j$ rows are fixed. Similar reasoning as in part (b) yields the following count
$$\frac{(q^n-q^j)(q^n-q^{j+1})\dots (q^n- q^{n-1})}{(q-1)^{n-j}}=\left(q^{\left(n(n-1)-j(j-1)\right)/2}\right)\left(\prod _{k=1} ^{n-j} \left[\begin{matrix}k\\1 \end{matrix}\right]_q\right).$$

(e) The statement follows from (d) because the correspondence between the ordered tuples of part (d) and the unordered ones is $(n-j)! :1$.

(f)  A path  from the minimum element to the maximum element in the lattice ${\cal V}(n,q)$ is a chain
of vector subspaces in $\FF^n$
\[W_0=\FF^0  \subset W_1 \subset W_2  \subset \ldots \subset W_n=\FF^n, \]
with  $W_k \in {\cal V} _k$.
Let $W \in {\cal V} _k$. The number of $(k+1)$-dimensional subspaces in $\FF^n$ which contains $W$ is
$\qbinom{n-k}{1}$,
since this number is the same as the number of linearly independent vectors in $\FF^n/W$, which is $(n-k)$-dimensional.
Hence the assertion follows.
\end{proof}

\section{The determinant of the incidence matrix between  ${\cal V}_1$ and  ${\cal V}_{n-1}$}\label{sect3}

We use the notation fixed in section~\ref{sect2}.
 A recurring theme in our work will be the occurrence of matrices of a special form, for which determinants are relatively easily computed. We find it useful to introduce a uniform notation for these matrices.

\begin{notation}\label{matrix phi}
Let  $\Phi(\nu, \al, \be)$ denote the matrix of size $\nu \times \nu$
with entries
$$
\phi _{ij} =
\begin{cases}
\alpha  &(i=j) \\
\beta &(i \neq j).
\end{cases}
$$
\end{notation}

\begin{lem} \label{lemphi}
\begin{enumerate}[\rm(a)]
\item
The determinant of $\Phi(\nu, \al, \be)$ is given by
$$
\det \Phi(\nu, \al, \be)=(\al -\be)^{\nu -1}(\nu \be +\al -\be).
$$
\item If $\alpha-\beta=\alpha'-\beta'$ then
$$\frac{\det \Phi(\nu, \al, \be)}{\det \Phi(\nu, \al', \be')}=\frac{\nu\be+\al-\be}{\nu\be'+\al'-\be'}.$$
\end{enumerate}
\end{lem}

\begin{proof}
Part (a) follows after performing convenient row and column operations on $\Phi(\nu, \al, \be)$ to transform the matrix to an almost diagonal form. Part (b) then follows from (a).
\end{proof}

\begin{defn}\label{defB}
In addition to the incidence matrix $A$ of Definition \ref{defA}, we consider  the
$N \times N$ matrix $B=(b_{ij})$ whose entries are
$$
b_{ij}=1-a_{ij}=
\begin{cases}
0  &(v_i \in v_j^{\perp}) \\
1 &(v_i \not \in v_j^{\perp}).
\end{cases}
$$
\end{defn}
\noindent As it will turn out, the determinant of $B$ is easier to compute than that of $A$ and we use the relation between $A$ and $B$ to complete our computation. Furthermore, both of these matrices carry deeper algebraic meaning, as we shall see in  section \ref{sect3}.

We begin with a few structural observations regarding the matrices $A$ and $B$.

\begin{lem} \label{row sum and column sum}
\begin{enumerate}[\rm(a)]
\item
Matrices $A$ and $B$ are symmetric.
\item
$\sum _{i=1}^N a_{ij}=\sum _{j=1}^N a_{ij}= \qbinom{n-1}{1}$.
\item
$\sum _{i=1}^N b_{ij}=\sum _{j=1}^N b_{ij}= \qbinom{n}{1}-\qbinom{n-1}{1}=q^{n-1}$.
\end{enumerate}
\end{lem}

\begin{proof}
Note that $v_i \in v_j ^{\perp}$  if and only if $v_j \in v_i ^{\perp}$,  which follows from the inclusion-reversing property of  dual spaces. This implies part (a).
The row sum of $A$ is equal to the number of  codimension 1 subspaces
in $\FF^{n-1}$ which contain $v_1$, and this is equal to the number of the 1-dimensional subspaces
in $v_1^{\perp} \cong \FF ^{n-1}$. Hence, part (b) follows.
Finally, since $N=\qbinom{n}{1}$,  part (c) follows as a consequence of the relations $b_{ij}=1-a_{ij}$.
\end{proof}

The following result shows the role played by the matrices $\Phi(\nu, \al, \be)$ in relation to $A$ and $B$.

\begin{lem} \label{description of A and B}
The following hold:
\begin{enumerate}[\rm(a)]
\item
$A^2=\Phi\left(N, \qbinom{n-1}{n-2}, \qbinom{n-2}{n-3}\right)$
\item
$B^2=\Phi(N, q^{n-1}, q^{n-2}(q-1))$
\item
$AB=\Phi(N,0,q^{n-2})$.
\end{enumerate}
\end{lem}

\begin{proof}
(a) The $(i,j)$-th entry of $A^2$ is $\sum _{k=1}^N a_{ik}a_{kj}$. Note that
$$
a_{ik}a_{kj}=
\begin{cases}
1  &(v_i, \ v_j \in v_k^{\perp}) \\
0 &(\text{otherwise}).
\end{cases}
$$
Hence, if $i=j$, the sum $\sum _{k=1}^N a_{ik}a_{kj}$ is equal to the number of codimension 1 subspaces in $\FF^n$ which contain $v_1$ and this number is $\qbinom{n-1}{n-2}$ because these subspaces are in bijection with codimension one subspaces of $v_1^\perp\simeq\FF^{n-1}$. If $i\neq j$, the sum $\sum _{k=1}^N a_{ik}a_{kj}$ is equal to the number of codimension 1 subspaces in $\FF^n$ which contain both $v_1$ and $v_2$. This number is $\qbinom{n-2}{n-3}$ because the codimension 1 subspaces in $\FF^n$ which contain both $v_1$ and $v_2$ are in bijection with codimension 1 subspaces in $ \{v_1, v_2\}^\perp\simeq\FF^{n-2}$. This  proves the assertion for $A^2$.

(b) The $(i,j)$-th entry of $B^2$ is $\sum _{k=1}^N b_{ik}b_{kj}$. For the diagonal entry of $B^2$ we have to count the number of the codimension 1 subspaces of $\FF^n$ which  do not contain $v_1$.  This number is $q^{n-1}$ since we have $\qbinom{n}{1}-\qbinom{n-1}{1}=q^{n-1}$.
To compute the off-diagonal entry of $B^2$ we use the inclusion-exclusion formula, since we have to count the number of the subspaces of $\FF^n$ of codimension 1 which  contain neither $v_1$ nor $v_2$.  The number of the subspaces in  $\FF^n$ of codimension 1 is $\qbinom{n}{1}$, and the number of the subspaces of codimension 1 which contain $v_1$ is $\qbinom{n-1}{1}$ and the same is true for $v_2$.
The number of the subspaces of codimension 1 which contain both $v_1$ and $v_2$ is $\qbinom{n-2}{1}$.  Hence,
$$\sum _{k=1}^N b_{ik}b_{kj}= \left[\begin{matrix}n\\1 \end{matrix}\right]_q - 2\left[\begin{matrix}n-1\\1 \end{matrix}\right]_q + \left[\begin{matrix}n-2\\1 \end{matrix}\right]_q=q^2(q-1).$$

(c) By the definition $A+B=\Phi(N, 1,1)$, which is the $N \times N$ matrix with 1 for all entries.  Hence $(A+B)B$ is the matrix which has the row sum of $B$ for all entries.
By Lemma~\ref{row sum and column sum}, this row sum is $q^{n-1}$.
Thus the diagonal entries of $AB$ are 0 and the off-diagonal entries are
equal to $q^{n-1}-q^{n-2}(q-1) = q^{n-2}$.
\end{proof}

 At this point, part (a) of Lemma \ref{description of A and B} together with the formula in Lemma \ref{lemphi} would allow us to complete the computation of $|\det(A)|$. It turns out, however, that it is easier to find $|\det(B)|$ first and utilize the relationship between the two determinants than to simplify the expression resulting from a direct approach. The following is the main result of this section.


\begin{thm} \label{main1}
For the matrices $A$ and $B=\Phi(N,1,1)-A$, we have
\begin{enumerate}[\rm(a)]
\item
$ \det (B^2)=q^{(n-2)N+n}$.
\item
$\left| \det B \right|=q^{((n-2)N+n)/2}$.
\item
$\left| \det A \right|=\left(q^{(n-2)(N-1)/2}\right)\qbinom{n-1}{1}$.
\end{enumerate}
\end{thm}

\begin{proof}
Lemmas~\ref{lemphi}(a) and~\ref{description of A and B}(b) imply that $\det (B^2)=q^{(n-2)N}(N(q-1)+1)$.  Now part (a) follows from the formula $N=(q^n-1)/(q-1)$ and part (b) follows immediately from (a).

Recall from Lemma \ref{description of A and B} the identities $A^2=\Phi\left(N, \qbinom{n-1}{n-2}, \qbinom{n-2}{n-3}\right)$ and $B^2=\Phi(N, q^{n-1}, q^{n-2}(q-1))$. Since $N=\qbinom{n}{1}$, $\qbinom{n-1}{n-2}- \qbinom{n-2}{n-3}=\qbinom{n-1}{1}- \qbinom{n-2}{1}=q^{n-2}$ and $q^{n-1}-q^{n-2}(q-1)=q^{n-2}$, it follows from Lemma~\ref{lemphi}(b) that 
\begin{eqnarray*}
\det(A^2):\det(B^2) &=& \left(N\cdot\frac{q^{n-2}-1}{q-1}+q^{n-2}\right):\left(Nq^{n-2}(q-1)+q^{n-2}\right)\\
&=& \left(\frac{(q^n-1)(q^{n-2}-1)}{(q-1)^2}+q^{n-2}\right):\left((q^n-1)q^{n-2}+q^{n-2}\right)\\
 &=& \frac{(q^n-1)^2}{(q-1)^2}: q^{2n-2}
 = \left(\qbinom{n-1}{1}\right)^2 : (q^{n-1})^2.
 \end{eqnarray*}
Taking the square root gives $\left| \det A \right| : \left| \det B \right|= \qbinom{n-1}{1} : q^{n-1}$. Together with part (b), this implies part (c) of the theorem.
\end{proof}



\begin{disc}
We can compute the determinant for $A$ also using the description of $AB$. From Lemmas~\ref{lemphi} and~\ref{description of A and B}, noting that
$N-1=\qbinom{n}{1}-1=q \qbinom{n-1}{1}=\left|\det \Phi(N, 0,1)\right|$, one gets
$\left|  \det (AB) \right|=(N-1)q^{N(n-2)}$, whence
$ \left|\det A \right|=(N-1)q^{(N(n-2)-n)/2}$.
This description for $\left|\det A \right|$ is slightly different from
Theorem~\ref{main1}. Of course, they are in fact the same since we have
\[N-1 =\left[\begin{matrix}n\\1 \end{matrix}\right]_q-1=q\left[\begin{matrix}n-1\\1 \end{matrix}\right]_q.\]
\end{disc}

\begin{disc} The result in Theorem \ref{main1} $($c$)$  recovers  the non-vanishing of one of the determinants involved in the definition of the strong Lefschetz property for ranked posets given in the Introduction. We recall that the vector space lattice has been proven to have the strong Lefschetz property in \cite{Ka}.
\end{disc}

\begin{ex}
The determinant computations below were obtained directly using Mathematica \cite{refMat}, independent of Theorem~\ref{main1} for $q=2, 3, 5$.

\begin{center}$
\begin{array}{ccccccccccccccccccccccccccc}
q=2      &&&&&&&              &&&&&&&&&&&&&&&                 &                &                 &                &
\end{array}
$

\noindent$
\begin{array}{|c||c|c|c|c|c|c|}
\hline
n       & 3             &    4              &   5              &               6  &  7              & 8     \\  \hline
N & 7 & 15 & 31& 63 & 127 & 255 \\ \hline
\det A  & 2^{3}\cdot 3   &   2^{14}\cdot 7   &   2^{45}\cdot 15  & 2^{124} \cdot 31  &  2^{315}\cdot63  & 2^{762}\cdot127  \\  \hline
\det B  & 2^{3}\cdot 2^{2}&  2^{14}\cdot 2^3  &   2^{45}\cdot 2^4 & 2^{124} \cdot 2^5 &  2^{315}\cdot2^6 & 2^{762}\cdot2^7   \\ \hline
\end{array}
$
\end{center}

\noindent $
\begin{array}{cccccccccccccccccccc}
q=3      &&&&&&&              &&&&&&&&&&&&
\end{array}
$
\quad
$
\begin{array}{ccccccccccc}
q=5     &&&&&&&              &&&
\end{array}
$

\noindent$
\begin{array}{|c||c|c|c|c|} \hline
n       & 3             &    4               &   5                       & 6                     \\  \hline
N & 13 & 40 & 121 & 364  \\ \hline
\det A  & 3^{6}\cdot 2^2   &  3^{39}\cdot 13  &   3^{180}\cdot 2^3\cdot 5  & 3^{726} \cdot 11^2     \\  \hline
\det B  & 3^{6}\cdot 3^{2} &  2^{39}\cdot 3^3  &   3^{180}\cdot 3^4        & 3^{726} \cdot 3^5       \\ \hline
\end{array}
$
\quad
$
\begin{array}{|c||c|c|} \hline
n       & 3                      &    4                \\  \hline
N & 31 & 156 \\ \hline
\det A  & 5^{15}\cdot 2 \cdot 3   &   5^{155}\cdot 31     \\  \hline
\det B  & 5^{15}\cdot 5^{2}        &  5^{155}\cdot 5^3    \\ \hline
\end{array}
$
\end{ex}

\smallskip

\section{The Hessian of the Macaulay dual generator for the Gorenstein algebra associated to the vector space lattice  }\label{sect4}

In this section, we relate the combinatorial data of section \ref{sect2} to algebraic invariants arising from a graded ring associated to the vector space lattice.

Recall that $N=\qbinom{n}{1}$. Consider the polynomial rings
$R=K[X_1, \ldots, X_N]$ and $Q=K[x _1, \ldots, x _N]$, where $K$
is a field of characteristic zero. Setting $x_i=\pa/\pa X_i$ allows one to view $R$ as a $Q$-module via the partial differentiation action of $Q$ on $R$ given by $x _i\circ f=\pa f/\pa X_i$, for  $f\in R$.

A bijection can be established between the set of variables in $R$ and the set $\PP ^{n-1}_{\FF}$ of vectors of length $n$ with entries in the field $\FF$ in which the first non-zero entry is $1$. We fix this bijection once and for all, so that  the variable $X_i$ corresponds to the vector $v_i \in  \PP ^{n-1}_{\FF}$. 

 We now outline the  construction given in \cite{MN} of a graded Artinian Gorenstein algebra associated to the vector space lattice. 
This uses the theory of Macaulay inverse systems, which provides a  correspondence between homogeneous polynomials in the ring $R$ and 
graded Artinian Gorenstein quotient algebras of $Q$.  For more details on Macaulay inverse systems the reader may consult \cite{Ge} and~\cite{IK}.

\begin{defn}\label{Macaulay correspondence}
For a homogeneous polynomial $F\in R$, the annihilator of $F$ in $Q$ is the ideal $I\subset Q$  defined by
\[\Ann _Q(F):=\{f \in Q\ \mid\  f \circ F=0\}.\]
If $I$ is an ideal of $Q$  the following set is the annihilator of $I$ in $R$:
\[\Ann _R(I):=\{F \in R\ \mid\  f \circ F=0, \  \forall f\in I \}.\]
\end{defn}

Let    $I    \subset   Q$   be a homogeneous ideal of  finite colength. It is well known that if  $Q/I$  is Gorenstein, then there exists a homogeneous form  $F \in R$  such that
$I = \Ann _Q (F)$.  On the other hand, if  $F \in  R$ is homogeneous, then
$I = \Ann_Q(F)$  is a homogeneous ideal and $Q/\Ann_Q(F)$ is an Artinian Gorenstein algebra.

The idea of constructing a Gorenstein algebra associated to the vector space lattice is that one can encode its combinatorial structure in a homogeneous polynomial of $R$ and then consider the graded Gorenstein quotient of $Q$ corresponding to it.

\begin{defn}\label{Mac dual}
We define the \emph{Macaulay dual generator} for the vector space lattice to be the following degree $n$ homogeneous polynomial in $R$
$$F_{{\mathcal V}(n,q)}=\sum_{ X_{i_1}X_{i_2}\ldots X_{i_n}\in \cal B} X_{i_1}X_{i_2}\ldots X_{i_n}.$$ In the sum, the sets of indices of the variables appearing in each monomial represent the subsets of $\cal V_1$ that form bases for $\FF^n$, namely:
\[ \cal B = \{ X_{i_1}X_{i_2}\ldots X_{i_n}\ \mid\ 1 \leq i_1 < i_2 < \ldots i_n \leq N\ \text{and}\ {\rm det}[v_{i_1}v_{i_2} \ldots v_{i_n}] \neq 0 \mbox{ in } \FF \}. \]
\end{defn}

\noindent The cardinality of the set $\cal B$ above is according to Proposition \ref{the number of paths}
$$
\card(\cal B)=s_{n,q}=\left(\frac{q^{n(n-1)/2}}{n!}\right)\left(\prod _{k=1} ^n \left[\begin{matrix} k\\1\end{matrix}\right]_q\right).
$$

 \begin{defn}
 Setting $I=\Ann_Q(F_{{\mathcal V}(n,q)})$ yields a graded Artinian  Gorenstein quotient ring $\cal A_{{\mathcal V}(n,q)}=Q/I$, which we call the \emph{Gorenstein  algebra associated to the vector space lattice}.   For simplicity, we write $\cal A$ for $\cal A_{{\mathcal V}(n,q)}$ henceforth, unless otherwise specified.
 \end{defn}

 This graded ring decomposes into homogeneous components  as follows$$\cal A=Q/I=\bigoplus_{i=0}^n (Q/I)_i=\bigoplus_{i=0}^n \cal A_i.$$

 One notices the similarity between the homogeneous decomposition of $\cal A$ and  the rank decomposition of $\cal V(n,q)$.  It is shown in~\cite[Lemma 1.48 and step 4 in the proof of Theorem 1.83]{HMMNWW} and~\cite[Lemma 4.1 and Theorem 4.2]{MN} that the non-zero monomials in $\cal A$ are in bijective correspondence with the elements of ${\cal V(n,q)}$ in such a way that the level set ${\cal V}_i$ corresponds to the monomials in the  graded component $\cal A_{i}$.
In particular, we have the following correspondences
\begin{eqnarray*}
\cal A_0 \ni  1 & \leftrightarrow  &  \FF ^0 \in {\cal V}_0 \\
\cal A_n \ni  g  & \leftrightarrow  &  \FF ^n \in {\cal V}_n,
\end{eqnarray*}
 where $g$ is a monomial of degree $n$ called a socle generator for $\cal A$. The socle of $\cal A$ is a 1-dimensional vector space, thus in  $\cal A$ $g$ is unique up to scalar. However any product of variables of $Q$ whose indices correspond to a basis of $V$ and can be chosen to be a representative for $g$.

Next we recall the algebraic counterpart of the Lefschetz properties defined for ranked posets in the Introduction, with the end goal of explicitly relating the incidence matrices of section \ref{sect2} with certain matrices arising from the Macaulay dual generator in Definition \ref{Mac dual}.

  Consider for some scalar values $a_1, \ldots, a_N\in K$ the linear form
 $$L=a_1x_1 + \ldots + a_Nx_N\in Q$$  and let $0\leq j\leq \lfloor n/2 \rfloor$. We set $\times L^{n-2j}: \cal A\to \cal A$ to be the $Q$-module homomorphism given by $x\mapsto L^{n-2j}x$. Restricting to the degree $j$ and $n-j$ homogeneous components of $\cal A$, we obtain the $K$-linear maps
\[\times L^{n-2j}: \cal A_j \to \cal A_{n-j}.\]
The motivation for considering such a map originally arises from the study of cohomology rings of  compact K\"{a}hler manifolds, where one can regard such a map as taking a class in cohomology and intersecting it with hyperplanes (represented by $L$) $n-2j$ times. 

Fixing the sets of monomials corresponding to elements of $\cal V_j$ and $\cal V_{n-j}$, respectively, as bases for $\cal A_j$ and $\cal A_{n-j}$ one can express the linear transformations $\times L^{n-2}$  as  matrices $M_j$. Note that $\dim_K\cal A_j=\dim_K \cal A_{n-j}$ since the bases for these vector spaces correspond to symmetric level sets $\cal V_j$ and $\cal V_{n-j}$ of $\cal V(n,q)$ which have the same size. Thus, it makes sense to consider $\det M_j$.

\begin{defn}\label{def Lefschetz}
Let $\cal A$ be any graded Gorenstein Artinian algebra. If there exist scalars $a_1,\ldots, a_N\in K$ such that the matrices $M_j$ representing the $K$-linear maps $\times L^{n-2j}: \cal A_j \to \cal A_{n-j}$ for $L=a_1x_1 + \ldots + a_Nx_N$ have $\det M_j\neq 0$ for all  $0\leq j\leq \lfloor n/2 \rfloor$, the algebra $\cal A$ is said to have the \emph{strong Lefschetz property}.
\end{defn}


We turn to our case of interest $\cal A=\cal A_{{\mathcal V}(n,q)}$ and  focus on a particular choice of linear form, $\ell = x_1+x_2+\ldots+ x_N$. We shall be particularly concerned with computing the determinant of the matrix that represents the map $\times \ell^{n-2}$.  Setting  $x_i^\perp=\prod_{v_j\in v_i^\perp} x_j$, consider the bases $\cal B_1=\{x_1,\ldots,x_N\}$ for $\cal A_1$ and $\cal B_{n-1}=\{x_1^\perp,\ldots, x_N^\perp\}$ for $\cal A_{n-1}$, which we shall call canonical bases, and let $M$ be the matrix that represents the linear transformation $\ell^{n-2}$ with respect to these fixed bases. 

\begin{ex}\label{ex2}
Let $q=2$ and $n=3$, which yield  $N=7$. We use the notation of Example \ref{ex1}.
A computation with {\em Macaulay2} \cite{M2} yields that the matrix representing $\times \ell: \cal A_1\to \cal A_2$ with respect to the bases
$$\begin{array}{rcl}
\cal B_1 &= & \{x_1, x_2, x_3,x_4,x_5,x_6,x_7\}\\
\cal B_2 & = & \{x_2x_4,  x_1x_4, x_3x_4, x_1x_2, x_2x_5, x_1x_6, x_3x_5 \}\\
\end{array}$$
is the matrix $M$ below, related to  the incidence matrix $A$ computed in example \ref{ex1} as follows:
$$
M=\begin{pmatrix} 0&2&0&2&0&2&0 \\ 2&0&0&2&2&0&0 \\ 0&0&2&2&0&0&2 \\ 2&2&2&0&0&0&0 \\ 0&2&0&0&2&0&2\\ 2&0&0&0&0&2&2\\ 0&0&2&0&2&2&0\end{pmatrix}=2A.
$$

\end{ex}

The following Theorem describes the precise relation between the incidence matrix $A$ of Definition \ref{defA} and the matrix describing multiplication by $\ell^{n-2}$.

\begin{thm}\label{multLn-2}
The matrix $M$ representing $\times \ell^{n-2}$ with respect to the standard bases for $\cal A_1$ and $\cal A_{n-1}$ is $M=t_{n-1,1,q}A$.
Hence,
$\left|\det M\right|=t_{n-1,1,q}^N\left | \det A\right |$.
\end{thm}

\begin{proof}
To find the entry of $M$ in the position indexed by the variable $x_i\in \cal A_1$ corresponding to $v_i$  and the basis element $x_j^\perp=\prod_{v_k\in v_j^\perp}x_k\in \cal A_{n-1}$ corresponding to the element $v_j^\perp\in \nu_{n-1}$,  we need to count the number of monomials $x_{k_1}x_{k_2}\ldots x_{k_{n-1}}$ in the expansion of
$\ell^{n-1}$ in the polynomial ring  which satisfy the following conditions:

\begin{itemize}
\item[(a)]
One of $x_{k_1}, x_{k_2}, \ldots, x_{k_{n-1}}$  is $x_i$.
\item[(b)]
${\rm Span}\langle v_{k_1}, v_{k_2}, v_{k_3}  \ldots, v_{k_{n-1}}   \rangle=v_j^\perp$ .
\end{itemize}

If $v_i\not \in v_j^\perp$ then clearly this number is zero. If  $v_i \in v_j^\perp$, then we need to count the number of  ordered $(n-1)$-tuples which form bases for  $v_j^\perp$ and contain $v_1$. By Proposition \ref{the number of paths} this number is
$$t_{n-1,1,q}=\left(q^{(n-1)(n-2)/2}\right)\left(\prod _{k=1}^{n-2}\left[\begin{matrix}k\\1\end{matrix}\right]_q\right).$$
Hence, it follows from Definition \ref{defA} that the matrix for $\times \ell^{n-2}$ is $t_{n-1,1,q}A$.
\end{proof}

It is shown in~\cite{refMW} that there is a close connection between the matrices representing $\times L^{n-2j}$ for $L=a_1x_1 + \ldots + a_Nx_N$  and the determinants of  higher analogues of the classical Hessian matrix of the Macaulay dual generator $F_{{\mathcal V}_(q,n)}$, evaluated at $X_1=a_1,\ldots, X_N=a_N$. For our purposes it suffices to consider the classical Hessian, as this corresponds to $\times L^{n-2}$ which we have been able to relate to the incidence matrix in Theorem \ref{multLn-2}.

\begin{defn}\label{defhes}
The \emph{Hessian matrix} of a polynomial $F\in R=K[X_1,\ldots,X_N]$ is the matrix of partial derivatives $$H(F)=\left(\frac{\pa^2 F}{\pa X_i\pa X_j}\right)_{1\leq i\leq N, 1\leq j\leq N}.$$
\end{defn}

We begin by describing the Hessian matrix in our running example.
\begin{ex}\label{ex3}
Let $q=2$ and $n=3$. We use the notation of Examples \ref{ex1} and \ref{ex2}.
The Macaulay dual generator, as introduced in Definition \ref{Mac dual}, is
\begin{align*}F_{{\mathcal V}(3,2)}&= X_1X_2X_4+X_1X_3X_4+X_2X_3X_4+X_1X_2X_5+X_1X_3X_5+X_2X_3X_5\\
&+X_2X_4X_5+X_3X_4X_5+X_1X_2X_6+X_1X_3X_6+X_2X_3X_6+X_1X_4X_6\\
& +X_3X_4X_6+X_1X_5X_6+X_2X_5X_6+X_4X_5X_6+X_1X_2X_7+X_1X_3X_7\\
& +X_2X_3X_7+X_1X_4X_7+X_2X_4X_7+X_1X_5X_7+X_3X_5X_7+X_4X_5X_7\\
&+ X_2X_6X_7+X_3X_6X_7+X_4X_6X_7+X_5X_6X_7.
\end{align*}
Note that this polynomial has $s_{3,2}=28$ terms, in accordance to the formula in Proposition \ref{the number of paths}. A computation with {\em Macaulay2} \cite{M2} yields that, after evaluating at $X_1=\ldots=X_7=1$, the Hessian matrix is $$H(F_{{\mathcal V}(3,2)}) \vert_{X_1=X_2 = \ldots = X_7=1}
=
 \bgroup\begin{pmatrix}0&
       4&
       4&
       4&
       4&
       4&
       4\\
       4&
       0&
       4&
       4&
       4&
       4&
       4\\
       4&
       4&
       0&
       4&
       4&
       4&
       4\\
       4&
       4&
       4&
       0&
       4&
       4&
       4\\
       4&
       4&
       4&
       4&
       0&
       4&
       4\\
       4&
       4&
       4&
       4&
       4&
       0&
       4\\
       4&
       4&
       4&
       4&
       4&
       4&
       0\\
       \end{pmatrix}=\Phi(7,0,4)\egroup.$$
\end{ex}

In the following we aim to understand this especially nice form of the Hessian matrix by describing the relation between the Hessian of the Macaulay dual generator of $\cal A$ and the matrices introduced in section \ref{sect2}.

\begin{lem}\label{commutdiag}
Let $F\in R=K[X_1,\ldots,X_N]$ be a homogeneous polynomial of degree $n$, let $a_1,\ldots, a_N\in K$  and consider the linear form $L=a_1\frac{\partial}{\partial X_1}+\ldots  +a_N\frac{\partial}{\partial X_N}\in Q$.
Then there is a commutative diagram
$$\xymatrixcolsep{4.5pc}
\xymatrix{\cal A_1 \otimes_K \cal A _1\ar[r]^-{ \mathbf{1}_{\cal A_1}\otimes_K(\times L^{n-2})} \ar@/_1.5pc/[rrr]_-{(n-2)!H(F)\vert _{X_1=a_1,\ldots, X_N=a_N}}&\cal A_1 \otimes_K \cal A _{n-1}\ar[r]^-{\mu}&\cal A_n\ar[r]^{\circ F}&K}
$$
where
\begin{enumerate}[\rm(a)]
\item $\mu$ denotes the internal multiplication on $\cal A$,
\item the map $\circ F$ maps $f\in \cal A_n\mapsto f\circ F\in K$ and
\item
$H(F)\vert _{X_1=a_1,\ldots, X_N=a_N}$ denotes the $K$-bilinear form $\cal A_1\otimes_K \cal A_1\to K$  represented with respect to the basis  $\{\frac{\partial}{\partial X_1}, \frac{\partial}{\partial X_2} ,...,  \frac{\partial}{\partial X_N}\}$ of $\cal A_1$ by the matrix in Definition \ref{defhes} evaluated at $X_1=a_1,\ldots, X_N=a_N$.
\end{enumerate}
\end{lem}

\begin{proof}
From  the proof of \cite[Theorem 4]{Wa}, \cite[Theorem 3.1]{refMW} or \cite[Theorem 3.76]{HMMNWW} we have the following identity:
$$L^{n-2}\frac{\pa}{\pa X_i}\frac{\pa}{\pa X_j}F(X)=(n-2)!\frac{\pa}{\pa X_i}\frac{\pa}{\pa X_j}F(X)\vert_{X_1=a_1,\ldots,X_N=a_N}.$$
The left side of the expression above can be viewed as the composition of the three maps in the top line of the diagram, applied to the element $\frac{\pa}{\pa X_i}\otimes \frac{\pa}{\pa X_j}\in \cal A_1 \otimes_K \cal A_1$. The right side of the displayed equality is the bottom map in the diagram evaluated at the same element. The commutative diagram represents this equality in visual form.
\end{proof}

To exploit the relations illustrated in the above diagram, we prove the following.
\begin{prop}\label{mu}
The matrix describing the natural (bilinear) multiplication map $\cal A_1\otimes_K \cal A_{n-1}\stackrel{\mu}{\longrightarrow}\cal A_n$ with respect to the canonical bases of $\cal A_1$,$\cal A_{n-1}$ and  $\cal A_n$, respectively, is the matrix $B$ introduced in Definition \ref{defB}.
\end{prop}

\begin{proof}
Since the squares of variables are in the ideal $I$, by \cite[Proposition 3.1]{MN}, we have that the action of $\mu$ on the pairs of basis elements is the following:
$$
\mu(x_i,x_j^\perp)=
\begin{cases}
0  &(x_i \ | \ x_j^\perp, \mbox{ equivalently } v_i\in v_j^\perp) \\
g &(x_i \not | x_j^\perp,\mbox{ equivalently } v_i\not\in v_j^\perp).
\end{cases}
$$
Clearly then $\mu$ is represented  as a bilinear form by $B$ with respect to the  bases $\cal B_1$ and $\cal B_{n-1}$ of  $\cal A_1$ and $\cal A_{n-1}$,  and the basis $\{ g \}$ for $\cal A_n$, where $g$ is a monomial generator of  $\cal A_n$.
\end{proof}

We are now ready to see how the Hessian relates to the matrices $A$ and $B$.

\begin{thm}\label{main2}
The Hessian matrix of $F_{{\mathcal V}_(q,n)}$ evaluated at $X_1=\ldots=X_n=1$ is
$$H(F_{{\mathcal V}_(q,n)})\vert_{X_1=\ldots=X_n=1} = \ \frac{t_{n-1,1,q}}{(n-2)!}AB.$$
\end{thm}

\begin{proof}
It follows from Lemma~\ref{commutdiag} that the matrix of the Hessian is $\frac{1}{(n-2)!}$ times the product of the matrices of $\mu$ and $\times \ell^{n-2}$. Propositions \ref{mu} and Theorem \ref{multLn-2}, which give that the matrix representing $\mu$ is $B$ and the matrix representing $\times \ell^{n-2}$ is $t_{n-1,1,q}A$ respectively now finish the proof.
\end{proof}

\begin{cor}\label{det Hessian}
  The Hessian matrix of the dual socle generator $F_{{\mathcal V}(n,q)}$ evaluated at $X_1=X_2 = \ldots = X_N=1$ is given by
$$
H(F_{{\mathcal V}_(q,n)})\vert_{X_1=\ldots=X_n=1}=\Phi(N, 0,t_{n,2,q}).
$$
Hence the absolute value of the determinant for this matrix is
$$
\big|\det H(F_{{\mathcal V}_(q,n)}) \vert_{X_1=X_2 = \ldots = X_N=1}\big|=(N-1)t_{n,2,q}^N.
$$
\end{cor}

\begin{proof} This follows from Theorem~\ref{main2} and Proposition~\ref{description of A and B}(c), after one notices
$$
\left(\frac{q^{n-2}}{(n-2)!}\right)\left(t_{n-1,1,q}\right)=
\left(\frac{q^{(n^2-n-2)/2}}{(n-2)!}\right)\left(\prod _{k=1}^{n-2} \left[\begin{matrix}k\\1\end{matrix}\right]_q\right)=t_{n,2,q}.
$$


The determinantal formula in Lemma \ref{lemphi}  finishes  the proof.
\end{proof}

We conclude the paper with a description of the zeroth Hessian of $F_{{\mathcal V}(n,q)}$ evaluated at $X_1=X_2 = \ldots = X_N=1$  which is by definition $F_{{\mathcal V}(n,q)}(1,1,\ldots,1)$  and its implications on the map $\ell^n: \cal A_0 \to \cal A_n$.

\begin{prop}\label{zeroHessian}
 Recall that $\ell=\frac{\partial}{\partial X_1}+\frac{\partial}{\partial X_2}+\ldots+\frac{\partial}{\partial X_N}\in Q$. Then
\begin{enumerate}[\rm(a)]
\item
the $K$-linear homomorphism  $ K F_{{\mathcal V}(n,q)}  \to K$ mapping $ F_{{\mathcal V}(n,q)} \mapsto \ell^n F_{{\mathcal V}(n,q)}$ is given by the formula
$$\ell ^nF_{{\mathcal V}(n,q)}=q^{\frac{1}{2}n(n-1)} \prod _{k=1} ^n \qbinom{k}{1}$$
\item
the homomorphism  $\times \ell ^n: \cal A_0 \to \cal A_n$ is given with respect to the bases $\cal B_0=\{1\}$ and $\cal B_n=\{g\}$ (where $g$ is any monomial in $\cal A_n$) by multiplication by the integer
$$q^{\frac{1}{2}n(n-1)}\prod _{k=1}^n \qbinom{k}{1}.$$
\end{enumerate}
\end{prop}

\begin{proof}
(a) The coefficient of a square-free monomial in  $\ell^n$  is $n!$,  so acting by partial differentiation $\ell^nF_{{\mathcal V}(n,q)}=n!F_{{\mathcal V}(n,q)}(1,1,\ldots, 1)$. Since the number of monomials in $F_{{\mathcal V}(n,q)}$ is $F_{{\mathcal V}(n,q)}(1,1,\ldots, 1)=s_{n,q}$, Proposition \ref{the number of paths} ($c$) proves the first assertion.

(b) Since the  maps in (a) and (b) are dual to each other by the theory of inverse systems, it follows  that $ \times \ell^n: \cal A _0 \to \cal A _n $ is given by multiplication by the same integer as the map in (a). \end{proof}

\begin{disc} Our results in Proposition \ref{zeroHessian} and Corollary \ref{det Hessian} recover via Lemma \ref{commutdiag}  the non-vanishing of two of the determinants involved in the definition of the strong Lefschetz property (Definition \ref{def Lefschetz}) of the Gorenstein algebra $\cal A$. This algebra has been proven to have the strong Lefschetz property in \cite{MN}.
\end{disc}

\section*{Acknowledgments}
This paper started when the third author visited the Department of Mathematics at the University of Nebraska-Lincoln in April 2014. He is grateful to Luchezar Avramov for making this visit possible and for his hospitality. We are also grateful to the anonymous referee for his/her careful reading of our paper.

\bigskip

\providecommand{\bysame}{\leavevmode\hbox to3em{\hrulefill}\thinspace}
\providecommand{\MR}{\relax\ifhmode\unskip\space\fi MR }
\providecommand{\MRhref}[2]{%
  \href{http://www.ams.org/mathscinet-getitem?mr=#1}{#2}
}
\providecommand{\href}[2]{#2}

\bigskip

\end{document}